\documentclass[a4paper,oneside,5pt]{article}

\usepackage{amsmath}%
\usepackage{amssymb}%
\usepackage[latin1]{inputenc}
\usepackage{enumitem}
\usepackage{euscript}
\usepackage{graphicx}
\usepackage[section]{placeins}
\usepackage{makeidx}

\usepackage{geometry}
\newtheorem{theorem}{Theorem}

\newtheorem{definition}[theorem]{Definition}

\newtheorem{lemma}[theorem]{Lemma}

\newenvironment{proof}[1][Proof]{\textbf{#1.} }{\ \rule{0.5em}{0.5em}}

\newcommand{\refeqn}[1]{(\ref{#1})}
\newcommand{\virgolette}[1]{``#1''}

\newcommand{\matr}[0]{\operatorname{Mat}}

\newcommand{\var}[0]{\operatorname{var}}

\begin{document}

\title{{\bf A de Finetti-type theorem for random-rotation-invariant continuous semimartingales}}
\author{ Francesco C. De Vecchi\thanks{Dipartimento di Matematica, Universit\`a degli Studi di Milano, via Saldini 50, Milano, Italy, \emph{email: francesco.devecchi@unimi.it, francesco.devecchi.fdv@gmail.com}}}
\date{}

\maketitle

\begin{abstract}
We provide a characterization of continuous semimartingales whose law is invariant with respect to predictable random rotations. In particular we prove that all such semimartingales are obtained by integrating a predictable process with respect to an independent $n$ dimensional Brownian motion.
\end{abstract}

\bigskip

\noindent \textbf{Keywords}: Invariant stochastic processes, Random rotations, Continuous semimartingales \\
\textbf{MSC numbers}: 60F17, 60G44, 60G46\\

\section{Introduction}\label{section_introduction}

The problem of characterizing random objects invariant with respect to some group of (deterministic or random) transformations was faced for the first time by de Finetti for describing the form of a sequence of infinite random variables invariant with respect to finite permutations (see \cite{deFinetti1937}). This result has been generalized in many ways, for example considering different settings (continuous time processes, non-commutative probability etc.) or new types of transformations (time translations, rotatability predictable transformations for processes etc.) so that this topic is now a classical research field in probability (see, e.g., \cite{Aldous1985,Kallenberg2005} for some reviews on the subject).\\

\noindent In this paper we characterize the continuous semimartingales invariant with respect to predictable random rotations. Denoting by
$\mathcal{F}^Z_t$ the natural filtration generated by the semimartingale $Z$ and by $O(n)$ the group of $n \times n$ orthogonal matrices and
using Einstein notation, we introduce the following definition.

\begin{definition}\label{definition_invariant}
Let $Z$ be a semimartingale taking values in $\mathbb{R}^n$. We say that $Z$ is invariant with respect to (predictable) random rotations if, for any process $B$ predictable with respect to $\mathcal{F}^Z_t$ and  taking values in $O(n)$, we have that the $\mathbb{R}^n$ semimartingale $Z'$, given by
$$Z'^i_t=\int_0^t{B^i_{j,s}dZ^j_s},$$
has the same law of $Z$.
\end{definition}

\noindent The prototype of a semimartingale invariant with respect to random rotations is the $n$ dimensional Brownian motion. Indeed this kind of invariance is an easy consequence of L\'evy characterization of Brownian motion (see \cite{DMU1}). Nevertheless this property is not peculiar of Brownian motion but is shared by many other semimartingales, such as homogeneous $\alpha$-stable L\'evy processes, Hermetian random matrices and other Markovian and non-Markovian semimartingales (see \cite{Albeverio2017}). \\
This invariance property is useful for explaining the relationship between Brownian motion and Riemannian geometry (for example the well-known relationship between $\mathbb{R}^n$ Brownian motion and the stochastic development of Brownian motion on a Riemannian manifold, see \cite{Emery,Elworthy1982,Elworthy2010}) or more generally for studying  L\'evy processes taking values in Riemannian manifolds (see \cite{Applebaum2000}). Furthermore the above invariance with respect to random rotations is a particular case of \emph{gauge symmetry} introduced in \cite{Albeverio2017} for extending the Lie symmetry analysis from the deterministic to the general stochastic setting (see also \cite{DMU2,DMU1,Gaeta2017} for the particular case of Brownian motion).\\
Although the invariance property stated in Definition \ref{definition_invariant} has many similarities with the invariance
 with respect to rotations already studied in the literature, there is a fundamental difference. Indeed, in the traditional theorems of rotatability of stochastic processes, the rotations usually act both on the space variables of $\mathbb{R}^n$ processes and on the time variable $t$ (see, e.g., \cite{Kallenberg2005}). For example, a classical result of this form is that, if a real process $X$ is such that for any $h >0$ the infinite sequence of random variables $\{X_{nh}\}_{n \in \mathbb{N}}$ is invariant with respect to finite deterministic rotations, then there exists a random variable $\sigma$ and an independent Brownian motion $W_t$ such that $X_t=\sigma W_t$. In Definition \ref{definition_invariant} the action of the rotations group is only on the space variable and the previous result does not hold. Nevertheless we are able to prove a generalization of this result for continuous semimartingales, which is a de Finetti type representation theorem.

\begin{theorem}\label{theorem_main}
Under suitable hypotheses on the continuous semimartingale $Z$ taking values in $\mathbb{R}^n$ (more precisely hypotheses {\bf A}, {\bf A1} and {\bf B} below) if $Z$ is invariant with respect to random rotations (according to Definition \ref{definition_invariant}) then there exists an $n$ dimensional Brownian motion $W$ and a predictable process $f_t$, independent of $W$, and both adapted with respect to $\mathcal{F}^Z_t$, such that $\int_0^t{f^2_sds} < +\infty$ almost surely and
$$Z^i_t=\int_0^t{f_sdW^i_s}.$$
\end{theorem}

\noindent In order to  describe more precisely the hypotheses of Theorem \ref{theorem_main} we start by recalling the notion of characteristics of a $\mathbb{R}^n$ semimartingale (see \cite{Jacod2003}): the two continuous predictable processes of bounded variation $b_t \in \mathbb{R}^n$ and $A_t \in \matr(n,n)$ are the characteristics of the continuous semimartingale $Z$ with respect to its natural filtration $\mathcal{F}^Z_t$ if
$Z^i_t-b_t$ is a local ($\mathcal{F}^Z_t$) martingale and $A^{ij}_t=[Z^i,Z^j]_t$. It is clear that $A_t$ is a symmetric semidefinite positive matrix increasing with respect to $t$, i.e. $A_t-A_s$ is semidefinite positive whenever $t \geq s$.\\
We introduce the following hypothesis on $Z$
\begin{itemize}
\item hypothesis {\bf A}: $A_t$ is almost surely absolutely continuous, i.e. there exists a predictable process $\tilde{A}_{t}$ taking value in the set of symmetric semidefinite positive matrices almost surely finite with respect the measure $\mathbb{P} \otimes dt$ such that
$$A^{ij}_t=\int_0^t{\tilde{A}^{ij}_sds};$$
\item hypothesis {\bf A1}: the process $\tilde{A}^{ij}_t$ of the hypothesis {\bf A} is such that $\tilde{A}^{ij}_t$ is a symmetric \emph{definite positive} (and not only semidefinite positive) matrix almost surely with respect to $\mathbb{P} \otimes dt$;
\item hypothesis {\bf B}: $b_t$ is almost surely absolutely continuous, i.e. there exists a predictable process $\tilde{b}_t$ such that $b^i_t=b_0+\int_0^t{\tilde{b}^i_sds}$.\\
\end{itemize}

\noindent We briefly discuss the relationship between the hypotheses of Theorem \ref{theorem_main}, namely the invariance of $Z$ with respect to random rotations, the request that $Z$ is a continuous semimartingale and the above hypotheses {\bf A}, {\bf A1} and {\bf B}.\\
Obviously, when the thesis of Theorem \ref{theorem_main} holds, $Z$ is a continuous semimartingale and hypotheses {\bf A} and {\bf B} are satisfied. It is well known (see \cite{Albeverio2017}) that there are non-continuous semimartingales which are invariant with respect to random rotations, so the continuity of $Z$ is a necessary hypothesis. Indeed there is a deeper reason for the continuity hypothesis: since the discontinuous processes have not the martingales representation property if we use only It\^o integration of the form $\int_0^t{H_sdZ_s}$, we think that a trivial equivalent of Theorem \ref{theorem_main}, replacing the integral with respect to Brownian motion with some integral with respect to some L\'evy process, does not hold in the discontinuous case. Finally the continuity of the process $Z$, and so of the Brownian motion, is a key ingredient in proving Lemma \ref{lemma_main2} below.\\
We remark that Hypothesis {\bf A1} is not necessary if we change the thesis, not requesting that the Brownian motion $W$ is measurable with respect to the natural filtration $\mathcal{F}^Z_t$ of $Z$. Furthermore if we require an invariance property with respect to random rotations stronger than Definition \ref{definition_invariant} we are able to modify the proof of Theorem \ref{theorem_main} without using hypothesis {\bf A1} and suitably enlarging the probability space where $Z$ is defined. This stronger invariance needs the law of $Z$ to be invariant with respect to any random rotation predictable with respect to any filtration $\mathcal{F}_t$ generated by $\mathcal{F}^Z_t$ and by an other filtration $\mathcal{G}_t$ independent of $\mathcal{F}^Z_t$. Until now, all the semimartingales for which we are able to prove the invariance according to Definition \ref{definition_invariant} satisfy this stronger notion of invariance. For this reason Theorem \ref{theorem_main} might hold without hypothesis {\bf A1} and requesting (only) the invariance with respect to Definition \ref{definition_invariant}.\\

\noindent Finally under suitable hypotheses on the process $f_t$ it is possible to prove that, if $Z$ satisfies the thesis of Theorem \ref{theorem_main}, then $Z$ is also invariant with respect to random rotations. For example if $f_t$ is measurable with respect to a filtration $\mathcal{F}^H_t$ generated by a $\mathbb{R}^h$  semimartingale  $H$ whose law is uniquely characterized by its characteristics and independent of $W$, we are able to use Theorem 3.18 of \cite{Albeverio2017} proving that $\int_0^t{f_tdW_t}$ is invariant with respect to random rotations.\\

\noindent The paper is organized into two sections. In Section \ref{section_preliminaries} we introduce some notations, concepts and results useful in the proof of Theorem \ref{theorem_main}. Section \ref{section_proof} contains some lemmas and the proof of our main result.

\section{Notations and preliminaries}\label{section_preliminaries}

In this section in order to fix the setting, we provide some notations and results about the random rotations invariance of semimartingales.\\
We consider the probability space $\Omega$ given by the Fr\'echet space $C^0(\mathbb{R}_+,\mathbb{R}^n)$ with the usual $\sigma$-algebra $\tilde{\mathcal{F}}$ of the Borel sets.  In order to have a filtration on $\Omega$ we fix a probability measure $\mathbb{P}$ such that the coordinate process $\omega(t)=Z_t$, where $\omega \in \Omega$, is a semimartingale.
In the following when we consider $\sigma$-algebra generated by some random variables or some processes we always mean the usual completed $\sigma$-algebra generated by these random variables or by these processes.\\
Let $B_t$ be a predictable process with respect to the filtration $\mathcal{F}^Z_t$ taking values in $O(n)$ and define the process
\begin{equation}\label{equation_1}
Z'^i_t=\int_0^t{B^i_{j,s}dZ^j_s}.
\end{equation}
The process $Z'_t$ defines a measurable map $\Lambda^B:\Omega \rightarrow \Omega'$ (where $\Omega'=C^0(\mathbb{R}_+,\mathbb{R}^n)$) in the following way
$$\Lambda^B(\omega)(t)=Z'_t(\omega).$$
We denote by $\mathbb{P}'=\Lambda^B_*(\mathbb{P})$ the pushforward of the measure $\mathbb{P}$ with respect to $\Lambda^B$. The canonical process $\omega'(t)$ (where $\omega' \in \Omega'$) with respect to the probability measure $\mathbb{P}'$ has exactly the same law of $Z'_t$. For this reason, in the following, with a slight abuse of notation, we identify the canonical process $\omega'(t)$ with the process $Z'_t$.\\
If $Z$ is invariant with respect to random rotations, for any $B$ as above we have $\mathbb{P}'=\mathbb{P}$. We say that $\Lambda^B$ is almost surely invertible if there exists a predictable measurable map $\Lambda'^B:\Omega' \rightarrow \Omega$, such that $\Lambda^B \circ \Lambda'^B=id_{\Omega'}$ almost surely with respect to $\mathbb{P}'$ and $\Lambda'^B \circ \Lambda^B=id_{\Omega}$ almost surely with respect to $\mathbb{P}$.\\
Using the notion of characteristics of a semimartingale introduced in Section \ref{section_introduction} we state the following theorem.

\begin{theorem}\label{theorem_preliminary}
Let $(b_t,A_t)$ be the characteristics of a continuous semimartingale $Z$. If $Z$ is invariant with respect to random rotations (according to Definition \ref{definition_invariant}) then
\begin{eqnarray}
b^{i}_t(\omega)&=&\int_0^t{B^i_{k,s}(\Lambda'^B(\omega))db^{k}_s(\Lambda'^B(\omega))},\\
A^{ij}_t(\omega)&=&\int_0^t{B^i_{k,s}(\Lambda'^B(\omega))B^j_{\ell,s}(\Lambda'^B(\omega))dA^{k\ell}_s(\Lambda'^B(\omega))},
\end{eqnarray}
almost surely with respect to the measure $\mathbb{P}$.
\end{theorem}
\begin{proof}
This is a special case of Theorem 3.8 in \cite{Albeverio2017}.
${}\hfill$\end{proof}

In the following if $\Lambda^B$ is almost surely invertible and $K(\omega)$ is a random variable defined on $\Omega$ we define
$$\Lambda^B_*(K)(\omega')=K(\Lambda'^B(\omega')).$$
A random variable $K$ is said to be \emph{invariant with respect to the action of $\Lambda^B$} if $\Lambda^B_*(K)=K$ almost surely with respect to $\mathbb{P}$.\\

If $Z$ is a local martingale, under the hypothesis {\bf A1} it is simple to construct the Brownian motion whose existence is stated in Theorem \ref{theorem_main}. Indeed, consider the square root $\sqrt{\tilde{A}}_t$ of the matrix $\tilde{A}_t$. Since the matrix $\tilde{A}_t$ is almost surely invertible with respect to the measure $\mathbb{P} \otimes dt$, the matrix $C=(\sqrt{\tilde{A}})^{-1}$ is defined almost surely with respect to the measure $\mathbb{P} \otimes dt$. Furthermore $C$ is integrable with respect to $Z$ and the integral
$$W^i_t=\int_0^t{C^i_{k,s}dZ^k_s}$$
is a Brownian motion. Indeed $W^i$ are local martingales and
$$
\left[W^i,W^j\right]_t=\int_0^t{C^i_{k,s}C^j_{\ell,s}\tilde{A}^{k\ell}_s ds}=\delta^{ij}t,
$$
thus, by L\'evy characterization, $W$ is an $n$ dimensional Brownian motion. If $B$ is a predictable process such that $\Lambda^B$ is invertible it is simple to study the action of $\Lambda^B$ on $W^i$. In particular if $W'$ is the Brownian motion obtained with the previous procedure from $Z'$ we have that
$$\Lambda^B_*(W^i_t)=\int_0^t{\Lambda^B_*(B^{-1,i}_{j,s})dW'^j_s}.$$

\section{Proof of the main theorem}\label{section_proof}

We start by proving the following two lemmas.

\begin{lemma}\label{lemma_main1}
Under the hypothesis {\bf A} and {\bf B}, if a continuous semimartingale $Z$ with characteristics $(b,A)$ is invariant with respect to random rotations then $b_t=0$ almost surely and there exists a predictable process $F_t$ such that
$$A^{ij}_t=F_t \delta^{ij}.$$
\end{lemma}
\begin{proof}
Under the hypothesis {\bf A} for any $\epsilon >0$ the matrix $\tilde{A}^{ij}+\epsilon I_n$ is almost surely (with respect to the measure $\mathbb{P} \otimes dt$) a symmetric strictly positive definite matrix. Using Proposition 1.8 of \cite{DaPrato1992}, there exists a predictable process $B_t$ taking values in $O(n)$ such that $B_t \cdot (\tilde{A}_t + \epsilon I_n) \cdot B^T_t$ is a diagonal matrix for any $t$ and almost surely. Thus $B_t \cdot  \tilde{A}_t \cdot B^T_t$ is a diagonal matrix for any $t$ and almost surely. By exploiting \refeqn{equation_1} this means that
$$A'^{ij}_t=[Z'^i,Z'^j]_t=\int_0^t{B^i_{k,s}B^j_{\ell,s}dA^{k\ell}_s}=\int_0^t{B^i_{k,s}B^j_{\ell,s}\tilde{A}^{k\ell}_sds},$$
is diagonal. Since $Z$ is invariant with respect to random rotations, the semimartingale $Z$ has the same law of $Z'$ given by \refeqn{equation_1}. In particular the quadratic variation matrix $A$ of $Z$ has the same law of the quadratic variation $A'$ of $Z'$. This means that $A^{ij}_t=0=A'^{ij}$ whenever $i \not = j$, i.e. $A^{ij}_t$ is diagonal.\\
On the other hand if we choose a constant process $B$ and $A^{ii}_t$ are not all almost surely equals we have that
$A'^{ij}_t=\int_0^t{B^{i}_kB^j_k\tilde{A}^{kk}_sds}$ is not almost surely identically equal to zero when $i \not = j$. This means that $A'$ has not the same law of $A$ and so $Z'$ has not the same law of $Z$. Thus $A^{ij}_t=F_t \delta^{ij}$ for some increasing, positive and absolutely continuous process $F$.\\
Suppose that $b^i_t$ is not almost surely zero: this means that $Z^i_t$ is not a local martingale with respect to the natural filtration
$\mathcal{F}^Z_t$. Under the hypothesis {\bf B} there exists a predictable random process $B_t$ such that $(B_t \cdot \tilde{b}_t)^i=0$ almost
surely. Since the characteristics of $Z'$ under the filtration $\mathcal{F}^Z_t$ (and not, in general, under the natural filtration
$\mathcal{F}^{Z'}_t$ of $Z'$) are given by $\hat{b}_t=\int_0^t{B_s \cdot \tilde{b}_s ds}$, this means that $Z'^i$ is a local martingale with
respect to the filtration $\mathcal{F}^Z_t$. On the other hand since $Z'$ is $\mathcal{F}^{Z'}_t$ adapted, and since the filtration
$\mathcal{F}^Z_t$ contains the filtration $\mathcal{F}^{Z'}_t$, $Z'^i$ is a local martingale also with respect to $\mathcal{F}^{Z'}_t$. This
means that the characteristic $b'_t$ of $Z'$ with respect to $\mathcal{F}^{Z'}_t$ is such that $b'^i_t=0$ almost surely and $Z'$ cannot have the
same law of $Z$. Thus we must have $b_t=0$ almost surely.
${}\hfill$\end{proof}\\

\begin{lemma}\label{lemma_main2}
Under the hypotheses {\bf A}, {\bf A1} and {\bf B} let $K$ be a random variable defined on $\Omega$ invariant with respect to the action $\Lambda^B$, for any $B$ such that $\Lambda^B$ is almost surely invertible. Then $K$ is independent of the Brownian motion $W$ constructed in Section \ref{section_preliminaries}.
\end{lemma}
\begin{proof}
If $Z$ is invariant with respect to random rotations and hypotheses {\bf A} and {\bf B} hold, by Lemma \ref{lemma_main1} $Z$ must be a local martingale with respect to the filtration $\mathcal{F}^Z_t$, and so the process $W$ defined in Section \ref{section_preliminaries} is a well defined Brownian motion.\\
We define a sequence of stopping times $\tau^h_k$ depending on the real parameter $h >0$. Setting $\tau^h_0=0$, $\tau^h_1$ is defined as follows
$$\tau^h_1=\inf \{t | \ \  \|W_t\| \geq h \},$$
where $\| \cdot \|$ is the Euclidian norm of $\mathbb{R}^n$, while the stopping times $\tau^h_k$ ($k \geq 2$) are defined by recursion as
$$\tau^h_k=\tau^h_{k-1}+\inf \{t | \ \ \|W_{\tau^h_{k-1}+t}-W_{\tau^h_{k-1}}\|>h \}.$$
Let $\bold{B}=(B_1,B_2,...) \in O(n)^{\infty}$ be a sequence of (deterministic) rotations in $O(n)$ and define
$$B^{\bold{B},h}_t=\sum_{k \in \mathbb{N}}B_k I_{(\tau^h_{k-1},\tau^h_k]}(t).$$
Since $\tau^h_k$ are predictable stopping times, the process $B^{\bold{B},h}_t$ is predictable, and $\Lambda^{\bold{B},h}:=\Lambda^{B^{\bold{B},h}}$ is invertible with inverse $\Lambda^{\bold{B}^{-1},h}$ defined by the random rotation
$$\tilde{B}^{\bold{B},h}_t=\sum_{k \in \mathbb{N}}B^{-1}_k I_{(\tau'^h_{k-1},\tau'^h_k]}(t) $$
where $\tau'^h_{k}$ are the stopping times defined on $\Omega'$ with the same definition of $\tau^h_k$ but using the transformed Brownian motion $W'$ instead of $W$.\\
In order to prove that $\Lambda^{\bold{B}^{-1},h}$, defined above, is actually the inverse of $\Lambda^{\bold{B},h}$ we exploit the fact that the map $\Lambda^{B'}$, defined by the random rotation $B'$, is the inverse of the map $\Lambda^{\bold{B},h}$ if and only if
\begin{equation}\label{equation_2}
\left(\int_0^t{B'^i_{j,s} dZ'^j_s } \right) \circ \Lambda^{\bold{B},h}=Z^i_t
\end{equation}
$\mathbb{P}$ almost surely. In particular $\Lambda^{\bold{B}^{-1},h}$ is the inverse of $\Lambda^{\bold{B},h}$ if and only if the relation \refeqn{equation_2} holds with $B'_s=\tilde{B}^{\bold{B},h}_t$.
If we are able to prove that $\tau'^h_k \circ \Lambda^{\bold{B},h}=\tau^h_k$, using the fact that
\begin{eqnarray*}
Z'_t \circ \Lambda^{\bold{B},h}&=&\int_0^t{B^{\bold{B},h}_{s} \cdot dZ_s}\\
&=& \sum_{k \in \mathbb{N}}B_{k} \cdot (Z_{\tau^i_k \wedge t}- Z_{\tau^i_{k-1} \wedge t}),
\end{eqnarray*}
we obtain
\begin{eqnarray*}
\left(\int_0^t{\tilde{B}^{\bold{B},h}_s \cdot dZ'_s } \right) \circ \Lambda^{\bold{B},h}&=&
\sum_{k \in \mathbb{N}} B^{-1}_k \cdot (Z'_{(\tau'^h_k \circ \Lambda^{\bold{B},h}) \wedge t
}-Z'_{(\tau'^h_{k-1} \circ \Lambda^{\bold{B},h}) \wedge t})\\
&=&\sum_{k \in \mathbb{N}} B^{-1}_k \cdot B_k \cdot (Z_{\tau^h_k \wedge t
}-Z_{\tau^h_{k-1} \wedge t})=Z_{t},
\end{eqnarray*}
proving in this way that equation \refeqn{equation_2} holds and thus that $\Lambda^{\bold{B}^{-1},h}$ is the inverse of $\Lambda^{\bold{B},h}$.\\
We now prove that $\tau'^h_k \circ \Lambda^{\bold{B},h}=\tau^h_k$. Using the fact that
$$W'_t \circ \Lambda^{\bold{B},h}=\int_0^t{B^{\bold{B},h}_{s} \cdot dW_s}=
\sum_{k \in \mathbb{N}}B_{k} \cdot  (W_{\tau^h_k \wedge t}-W_{\tau^h_{k-1} \wedge t}),$$
we have that $\tau'^h_{1} \circ \Lambda^{\bold{B},h} \leq \tau^h_1$ since
$$\|(W' \circ \Lambda^{\bold{B},h})_{\tau^h_1} \| =\| B_1 \cdot W_{\tau^h_1}\|=\| W_{\tau^h_1} \|=h.$$
Using this result we have
$$\tau'^h_1 \circ \Lambda^{\bold{B},h}=\inf \{t \leq \tau^1_h | \ \  \|W'_t \circ \Lambda^{\bold{B},h}\| \geq h \}=\inf \{t \leq \tau_1^h | \ \  \|B_1 \cdot W_t\| \geq h \}=\inf \{t \leq \tau^1_h | \ \  \|W_t\| \geq h \}=\tau^h_1,$$
and, with analogous reasoning we can prove that $\tau'^h_k \circ \Lambda^{\bold{B},h}=\tau^h_k$. Therefore $\Lambda^{\bold{B}^{-1},h}$ is the inverse of $\Lambda^{\bold{B},h}$ and the stopping times $\tau^h_k$ are invariant with respect to $\Lambda^{\bold{B},h}$ for any $\bold{B} \in O(n)^{\infty}$.\\
In order to prove the lemma we introduce a $\sigma$-algebra $\mathcal{G}^h \subset \mathcal{F}^Z$ generated by
$$\mathcal{G}^h=\bigvee_{k \in \mathbb{N}} \sigma\left(W_{\tau^h_k}\right)=\bigvee_{k \in \mathbb{N}}\sigma\left(W_{\tau^h_k}-W_{\tau^h_{k-1}} \right).$$
Using the explicit expression of $B^{\bold{B},h}$ and the invariance properties of $\tau^h_k$ we have
$$\Lambda^{\bold{B},h}_*(W^i_{\tau_k}-W^i_{\tau_{k-1}})=B^{-1,i}_{k,j}(W^j_{\tau'^h_k}-W^j_{\tau'^{h}_{k-1}}).$$
Since $B_i$ are invertible matrices we have that $\Lambda^{\bold{B},h}(\mathcal{G}^h)=\mathcal{G}'^h$ where
$\mathcal{G}'^h=\bigvee_{k \in \mathbb{N}}\sigma\left(W'_{\tau'^h_k}\right)$.\\
Given a bounded continuous function $f:\mathbb{R} \rightarrow \mathbb{R}$, we define
$$K^{f,h}=\mathbb{E}[f(K)|\mathcal{G}^h]=\mathbb{K}^{f,h}(\Delta^hW_1,\Delta^hW_2,...)$$
where $\mathbb{K}:\mathbb{R}^{\infty} \rightarrow \mathbb{R}$ is a measurable map and $\Delta^hW_i=W_{\tau^h_i}-W_{\tau^h_{i-1}}$. By the
explicit expression of the inverse of $\Lambda^{\bold{B},h}$ and the invariance property of $\tau^h_k$ we have that
$$\Lambda^{\bold{B},h}(K^{f,h})=\mathbb{K}^{f,h}(B^{-1}_1 \cdot \Delta^h W'_1, B^{-1}_2 \cdot \Delta^h W'_2,... ). $$
On the other hand
\begin{eqnarray*}
\Lambda^{\bold{B},h}_*(K^{f,h})&=&\Lambda^{\bold{B},h}_*(\mathbb{E}_{\mathbb{P}}[f(K)|\mathcal{G}^h])\\
&=&\mathbb{E}_{\Lambda^{\bold{B},h}_*(\mathbb{P})}[\Lambda^{\bold{B},h}_*(f(K))|\Lambda^{\bold{B},h}(\mathcal{G}^h)].
\end{eqnarray*}
If $Z$ is invariant with respect to random rotations and thus $\Lambda^{\bold{B},h}(\mathbb{P})=\mathbb{P}$, and $K$ satisfies the hypotheses of the lemma we have
$$\Lambda^{\bold{B},h}_*(K^{f,h})=\mathbb{E}_{\mathbb{P}}[f(K)|\mathcal{G}'^h]=\mathbb{K}^{f,h}(\Delta^hW'_1,\Delta^hW'_2,...).$$
This means that, for any $\bold{B} \in O(n)^{\infty}$,
$$\mathbb{K}^{f,h}(\Delta^hW'_1,\Delta^hW'_2,...)=\mathbb{K}^{f,h}(B^{-1}_1 \cdot \Delta^h W'_1, B^{-1}_2 \cdot \Delta^h W'_2,... ).$$
Since the previous  equality holds for any $\bold{B} \in O(n)^{\infty}$ the random variable $K^{f,h}$ depends only on the random variables
$\|\Delta^h W'_k\|$. On the other hand, by definition of $\tau'^h_k$ we have $\|\Delta^h W'_k\|=h $, and thus the random variable $K^{f,h}$
depends only on the deterministic parameter $h$, i.e. it is a constant. Therefore the random variable $f(K)$ is independent in mean of the
$\sigma$-algebra $\mathcal{G}^h$, but since $f$ is any continuous function we have that $K$ is independent
of $\mathcal{G}^h$. \\
The last step is to prove that we can approximate the Brownian motion $W$ using $\mathcal{G}^h$ measurable random variables with respect
to almost surely convergence. This is sufficient for proving the lemma. Indeed, suppose that $g:\mathbb{R}^k \rightarrow \mathbb{R}$ is a
 continuous bounded function, and let $f$ be as above; if $W^{h_n}_{r}$ is a suitable sequence of $\mathcal{G}^{h_n}$ measurable random variables such that $W^{h_n}_{r} \rightarrow W_{t_r}$ almost surely, then
\begin{eqnarray*}
\mathbb{E}[f(K)g(W_{t_1},...,W_{t_r})]&=&\mathbb{E}[\lim_{n \rightarrow + \infty}f(K)g(W^{h_n}_{1},...,W^{h_n}_{r})]\\
&=&\lim_{n \rightarrow + \infty}\mathbb{E}[f(K)g(W^{h_n}_{1},...,W^{h_n}_{r})]=\mathbb{E}[f(K)]\lim_{n \rightarrow + \infty}\mathbb{E}[g(W^{h_n}_{1},...,W^{h_n}_{r})]\\
&=&\mathbb{E}[f(K)]\mathbb{E}[\lim_{n \rightarrow + \infty}g(W^{h_n}_{1},...,W^{h_n}_{r})]=\mathbb{E}[f(K)]\mathbb{E}[g(W_{t_1},...,W_{t_r})],
\end{eqnarray*}
and this ensures that $K$ is independent of the Brownian motion $W$.\\
In the following we prove that $W_1$ can be approximated by $\mathcal{G}^h$ measurable random variables. The general case is a simple extension. First of all we note that $\tau^h_{k}-\tau^h_{k-1}$ form a sequence of independent identically distributed random variables, since they depend all in the same way on the increments $W_{\tau^h_{k-1}+t}-W_{\tau^h_{k-1}}$ which are independent and identically distributed. Furthermore by the rescaling property of Brownian motion, we have that the law of $\tau^h_{k}-\tau^h_{k-1}$ coincides with the law of $h^2 \tau^1_1$. On the other hand (see \cite{Takeuchi1985}) $\tau^1_1$ is an $L^2$ random variable with mean $\mu^n=\mathbb{E}[\tau^1_1]=\frac{1}{n}$ and variance
$\sigma^n=\var(\tau^1_1)=\frac{2}{n^2(n+2)}$. Let $k^h=\left\lfloor \frac{1}{\mu^n h^2} \right\rfloor \in \mathbb{N}$ be the maximum integer less then $\frac{1}{\mu^n h^2}$. Thus we have
\begin{eqnarray*}
\mathbb{E}[\tau^h_{k^h}]&=&\sum_{i=1}^{k^h}\mathbb{E}[\tau^h_{i}-\tau^h_{i-1}]=\mu^n h^2 \left\lfloor \frac{1}{\mu^n h^2} \right\rfloor \stackrel{h \rightarrow 0}{\longrightarrow} 1\\
\var(\tau^h_{k^h})&=&\sum_{i=1}^{k^h}\var(\tau^h_i-\tau^h_{i-1})= \sigma^n h^4 \left\lfloor \frac{1}{\mu^n h^2} \right\rfloor \stackrel{h \rightarrow 0}{\longrightarrow} 0.
\end{eqnarray*}
Since $\tau^h_{k^h}$ converges in probability to $1$ as $h \rightarrow 0$, choosing a suitable subsequence $h_k \rightarrow 0$ we have that $\tau^h_{k^h}$ converges to $1$ almost surely. Since the Brownian motion is almost surely continuous we have that $W_{\tau^{h_n}_{k^{h_n}}} \rightarrow W_1$ almost surely. So $W_1$ can be approximated by an almost surely convergent sequence of $\mathcal{G}^h$ random variables and the thesis is proved.
${}\hfill$\end{proof}\\

\noindent \begin{proof}[Proof of Theorem \ref{theorem_main}]
By Lemma \ref{lemma_main1} the matrix $A$ of quadratic covariation of $Z$ is of the form $A^{ij}_t=F_t \delta^{ij}$, where $F_t=\int_0^t{f^2_sds}$ for some predictable random process $f_t$. If $B_t$ is any predictable process taking values in $O(n)$ and such that $\Lambda^B$ is invertible, the quadratic variation matrix $A'$ of $Z'$ is
$$A'^{ij}_t=\int_0^t{\Lambda^B_*(B^{i}_{k,s}B^j_{\ell,s}\delta^{k\ell}f^2_s)ds}=\Lambda^B_*(F_t)\delta^{ij}.$$
By Theorem \ref{theorem_preliminary} we have
$$F_t(\omega')=\Lambda^B_*(F_t)(\omega'),$$
therefore the random process $F$ is invariant with respect to the action of $\Lambda^B$ for $B$ as above. If we fix some times $t_1,...,t_k$ and a continuous function $g:\mathbb{R}^k \rightarrow \mathbb{R}$, we can apply Lemma \ref{lemma_main2} to the random variable $g(F_{t_1},...,F_{t_k})$ proving that it is independent of the Brownian motion $W$. Since $k \in \mathbb{N}$, $t_i \in \mathbb{R}_+$ and the continuous function $g$ are arbitrary we have proved that the process $F$ is independent of the Brownian motion $W$. This implies that the process $f$ is independent of the Brownian motion $W$. Since, by the construction of $W$, $Z^i_t=\int_0^t{f_sdW^i_s}$, the theorem is proved.
${}\hfill$\end{proof}

\section*{Acknowledgements}

The author would like to thank Prof. Paola Morando and Prof. Stefania Ugolini for their  useful comments, suggestions and corrections of the first draft of the paper. This work was supported by Gruppo Nazionale Fisica Matematica (GNFM-INdAM) through the grant: \virgolette{Progetto Giovani, Symmetries and reduction for differential equations: from the deterministic to the stochastic case}.

\bibliographystyle{plain}
\bibliography{rotation(1)}

\def\cprime{$'$} \def\cprime{$'$} \def\cprime{$'$} \def\cprime{$'$}
  \def\cprime{$'$}
\begin{thebibliography}{10}

\bibitem{Albeverio2017}
Sergio Albeverio, Francesco~C. De~Vecchi, Paola Morando, and Stefania Ugolini.
\newblock Symmetries and invariance properties of stochastic differential
  equations driven by semimartingales with jumps.
\newblock {\em arXiv preprint arXiv:1708.01764}, 2017.

\bibitem{Aldous1985}
David~J. Aldous.
\newblock Exchangeability and related topics.
\newblock In {\em \'Ecole d'\'et\'e de probabilit\'es de {S}aint-{F}lour,
  {XIII}---1983}, volume 1117 of {\em Lecture Notes in Math.}, pages 1--198.
  Springer, Berlin, 1985.

\bibitem{Applebaum2000}
D.~Applebaum and A.~Estrade.
\newblock Isotropic {L}\'evy processes on {R}iemannian manifolds.
\newblock {\em Ann. Probab.}, 28(1):166--184, 2000.

\bibitem{DaPrato1992}
Giuseppe Da~Prato and Jerzy Zabczyk.
\newblock {\em Stochastic equations in infinite dimensions}, volume~44 of {\em
  Encyclopedia of Mathematics and its Applications}.
\newblock Cambridge University Press, Cambridge, 1992.

\bibitem{deFinetti1937}
Bruno de~Finetti.
\newblock La pr\'evision : ses lois logiques, ses sources subjectives.
\newblock {\em Ann. Inst. H. Poincar\'e}, 7(1):1--68, 1937.

\bibitem{DMU2}
Francesco~C. De~Vecchi, Paola Morando, and Stefania Ugolini.
\newblock Reduction and reconstruction of stochastic differential equations via
  symmetries.
\newblock {\em J. Math. Phys.}, 57(12):123508, 22, 2016.

\bibitem{DMU1}
Francesco~C. De~Vecchi, Paola Morando, and Stefania Ugolini.
\newblock Symmetries of stochastic differential equations: {A} geometric
  approach.
\newblock {\em J. Math. Phys.}, 57(6):063504, 17, 2016.

\bibitem{Elworthy1982}
K.~David Elworthy.
\newblock {\em Stochastic differential equations on manifolds}, volume~70 of
  {\em London Mathematical Society Lecture Note Series}.
\newblock Cambridge University Press, Cambridge-New York, 1982.

\bibitem{Elworthy2010}
K.~David Elworthy, Yves Le~Jan, and Xue-Mei Li.
\newblock {\em The geometry of filtering}.
\newblock Frontiers in Mathematics. Birkh\"auser Verlag, Basel, 2010.

\bibitem{Emery}
Michel {\'E}mery.
\newblock {\em Stochastic calculus in manifolds}.
\newblock Universitext. Springer-Verlag, Berlin, 1989.
\newblock With an appendix by P.-A. Meyer.

\bibitem{Gaeta2017}
Giuseppe Gaeta.
\newblock Symmetry of stochastic non-variational differential equations.
\newblock {\em Physics Reports}, 2017.

\bibitem{Jacod2003}
Jean Jacod and Albert~N. Shiryaev.
\newblock {\em Limit theorems for stochastic processes}, volume 288 of {\em
  Grundlehren der Mathematischen Wissenschaften [Fundamental Principles of
  Mathematical Sciences]}.
\newblock Springer-Verlag, Berlin, second edition, 2003.

\bibitem{Kallenberg2005}
Olav Kallenberg.
\newblock {\em Probabilistic symmetries and invariance principles}.
\newblock Probability and its Applications (New York). Springer, New York,
  2005.

\bibitem{Takeuchi1985}
Junji Takeuchi.
\newblock Some results on {B}essel processes.
\newblock {\em Proc. Japan Acad. Ser. A Math. Sci.}, 61(9):294--297, 1985.

\end{thebibliography}

\end{document}